\documentclass[12pt]{amsart}

\usepackage{amsfonts,amssymb,amsthm,eucal,amsmath,verbatim}
\allowdisplaybreaks

\usepackage{etoolbox}
\makeatletter
\patchcmd{\@settitle}{\uppercasenonmath\@title}{}{}{}
\patchcmd{\@setauthors}{\MakeUppercase}{}{}{}
\makeatother

\newtheorem{thm}{Theorem}
\newtheorem{cor}[thm]{Corollary}

\newcommand{\R}{\mathbb{R}}

\newcommand{\E}{\mathbb{E}}

\newcommand{\inprod}[2]{\left\langle #1, #2 \right\rangle}

\renewcommand{\P}{\mathbb{P}}

\newcommand{\s}{\mathbb{S}}

\begin{document}

\title[Projections of probability distributions]{\sc Projections of  probability distributions:\\
A measure-theoretic Dvoretzky theorem}
\author{\sc Elizabeth Meckes}
\thanks{Research supported by an American Institute of Mathematics 
Five-year Fellowship and NSF grant DMS-0852898.}

\begin{abstract}
Many authors have studied the phenomenon of typically Gaussian
marginals of high-dimensional random vectors; e.g., for a 
probability measure on $\R^d$, under mild conditions, 
most one-dimensional marginals
are approximately Gaussian if $d$ is large.  In earlier work, the
author used entropy techniques and Stein's method to show that this
phenomenon persists in the bounded-Lipschitz distance 
for $k$-dimensional marginals of $d$-dimensional
distributions,  if $k=o(\sqrt{\log(d)})$.  In this
paper, a somewhat different approach is used to show that the 
phenomenon persists if $k<\frac{2\log(d)}{\log(\log(d))}$, and
  that this estimate is best possible.  
\end{abstract}

\maketitle


\section{Introduction}

The explicit study of typical behavior of the margins of
high-dimensional probability measures goes back to Sudakov \cite{Sud},
although some of the central ideas appeared much earlier; e.g., the
1906 monograph \cite{Bor} of Borel, which contains the first rigorous
proof that projections of uniform measure on the $n$-dimensional
sphere are approximately
Gaussian for large $n$.  Subsequent major contributions
were made by Diaconis and Freedman \cite{DF},
von Weizs\"acker \cite{vonW}, Bobkov \cite{Bob}, and
Klartag \cite{Kla1}, among others.  The objects of study are a random
vector $X\in\R^d$ and its projections onto subspaces; the central
problem here is to show that for most subspaces, the resulting
distributions are about the same, approximately Gaussian, and moreover
to determine how large the dimension $k$ of the subspace may be
relative to $d$ for this phenomenon to persist.  This aspect in
particular of the problem was addressed in earlier work \cite{JOTP} of
the author.  In this paper, a  different approach is presented
to proving the main result of \cite{JOTP}, which, in addition to being
technically simpler and perhaps more geometrically natural, also gives
a noticable quantiative improvement.  The result shows that the
phenomenon of typical Gaussian marginals persists under mild
conditions for $k<\frac{2\log(d)}{\log( \log(d))}$, 
as opposed to
the results of \cite{JOTP}, which requires
$k=o(\sqrt{\log(d)})$ (note
that a misprint in the abstract of that paper claimed that
$k=o\left(\log(d)\right)$ was sufficient).  

The fact that typical $k$-dimensional projections of 
 probability measures on $\R^d$ are
approximately Gaussian when $k<\frac{2\log(d)}{\log(
  \log(d))}$ can be viewed as a measure-theoretic version of a
famous theorem of Dvoretzky \cite{Dvo}, V.\ Milman's proof of which
\cite{Mil} shows that for $\epsilon>0$ fixed and $\mathcal{X}$ a
$d$-dimensional Banach space, typical $k$-dimensional subspaces
$E\subseteq \mathcal{X}$ are $(1+\epsilon)$-isomorphic to a Hilbert
space, if $k\le C(\epsilon)\log(d)$.  (This is the usual formulation,
although one can give a dual formulation in terms of projections and
quotient norms rather
than subspaces.) These results should be viewed as analogous, in
the following sense: in both cases, an
additional structure is imposed on $\R^n$ (a norm
in the case of Dvoretzky's theorem; a probability measure in the present
context); in either case, there is a particularly nice way to do this 
(the Euclidean norm and the Gaussian distribution, respectively).  The
question is then: if one projects an arbitrary norm or probability
measure onto lower dimensional subspaces, does it tend to resemble this
nice structure?  If so, by how much must one reduce the dimension in order
to see this phenomenon?

Aside from the philosophical
similarity of these results, they are also similar in that additional
natural geometric assumptions lead to better behavior under
projections.  The main result of Klartag \cite{Kla2} shows that if the
random vector $X\in\R^d$ is assumed to have a log-concave
distribution, then typical marginals of the distribution of $X$ are
approximately Gaussian even when $k=d^{\epsilon}$ (for a specific
universal constant $\epsilon\in(0,1)$).  This should be compared in
the context of Dvoretzky's theorem to, for example, the result of
Figiel, Lindenstrauss and V.\ Milman \cite{flm} showing that if a
$d$-dimensional Banach space $\mathcal{X}$ has cotype
$q\in[2,\infty)$, then $\mathcal{X}$ has subspaces of dimension of the
  order $d^{\frac{2}{q}}$ which are approximately Euclidean; or the result
of Szarek \cite{Sza} showing that if $\mathcal{X}$ has bounded volume
ratio, then $\mathcal{X}$ has nearly Euclidean subspaces of dimension
$\frac{d}{2}$.  One interesting difference in the measure-theoretic
context from the classical context is that, for measures, it is possible
to determine {\em which} subspaces have approximately Gaussian projections 
under symmetry assumptions on the measure (see M.\ Meckes \cite{MM}); 
there is no known method to find explicit almost Euclidean subspaces of
Banach spaces, even under natural geometric
assumptions such as symmetry properties.

Following the statements of the main
results below, an example is given to show that the estimate
$k<\frac{2\log(d)}{\log(\log(d))}$ is best possible in 
the metric used here.

Before formally stating the results, some notation and context are 
needed.  The Stiefel manifold
$\mathfrak{W}_{d,k}$ is defined by 
$$
\mathfrak{W}_{d,k}:=\{\theta=(\theta_1,\ldots,\theta_k):\theta_i\in\R^d,
\inprod{\theta_i}{\theta_j}=\delta_{ij} \,\forall\, 1\le i,j\le k\},$$
with metric
$\rho\big(\theta,\theta^\prime\big)=\left[\sum_{j=1}^k|\theta_j-
  \theta_j^\prime|^2\right]^{1/2}$.  The manifold $\mathfrak{W}_{d,k}$
posseses a rotation-invariant (Haar) probability measure.

Let $X$ be a random vector in $\R^d$ and let $\theta\in\mathfrak{W}_{d,k}$.
Let
$$X_\theta:=\big(\inprod{X}{\theta_1},\ldots,\inprod{X}{\theta_k}\big);$$
that is, $X_\theta$ is the projection of $X$ onto the span of $\theta$.
Consider also the ``annealed'' version $X_\Theta$ for $\Theta
\in\mathfrak{W}_{d,k}$ distributed according to Haar measure and independent
of $X$.
The notation $\E_X[\cdot]$ is used to denote expectation with respect to 
$X$ only; that is, $\E_X[f(X,\Theta)]=\E\left[f(X,\Theta)\big|\Theta\right].$
When $X_\Theta$ is being thought of as conditioned on $\Theta$ with 
randomness coming from $X$ only, it is written $X_\theta$.
The following results describe the behavior of the random variables $X_\theta$
and $X_\Theta$.  In what follows, $c$ and $C$ are used to denote universal
constants which need not be the same in every appearance.

\begin{thm}\label{meanX}
Let $X$ be a random vector in $\R^n$, with $\E X=0$, $\E\left[|X|^2
\right]=\sigma^2d$, and let $A:=\E\big||X|^2\sigma^{-2}-d\big|$.  
If $\Theta$ is a random point of $\mathfrak{W}_{d,k}$, $X_\Theta$
is defined as above, and $Z$ is a standard Gaussian random vector, then
$$d_{BL}(X_\Theta,\sigma Z)\le \frac{\sigma[\sqrt{k}(A+1)+ k]}{d-1}.$$
\end{thm}

\begin{thm}\label{concX}Let $Z$ be a standard Gaussian random vector.  Let
\[B:=\sup_{\xi\in\s^{d-1}}\E\inprod{X}{\xi}^2.\]
For $\theta\in\mathfrak{W}_{d,k}$, let
$$d_{BL}(X_\theta,\sigma
  Z)=\sup_{\max(\|f\|_\infty,|f|_L)\le 1}\left|\E\Big[ f(\inprod{X}{\theta_1},\ldots,\inprod{X}{
\theta_k})\big|\theta\right]-
  \E f(\sigma Z_1,\ldots,\sigma Z_k)\Big|;$$ that is,
  $d_{BL}(X_\theta,\sigma Z)$ is the conditional bounded-Lipschitz distance
  from $X_\Theta$ to $\sigma Z$, conditioned on $\Theta$.  Then
if $\P_{d,k}$ denotes the Haar measure on $\mathfrak{W}_{d,k}$,
$$\P_{d,k}\left[\theta:\big|d_{BL}(X_\theta,\sigma Z)-
\E d_{BL}(X_\theta,\sigma Z)\big|>\epsilon\right]\le Ce^{-\frac{cd\epsilon^2}{B}}.$$

\end{thm}

\begin{thm}\label{dist-bdX}With notation as in the previous theorems, 
\[\E d_{BL}(X_\theta,\sigma Z)
\le C \left[\frac{(kB+B\log(d))B^{\frac{2}{9k+12}}}{(kB)^{\frac{2}{3}}
d^{\frac{2}{3k+4}}}+\frac{\sigma[\sqrt{k}(A+1)+ k]}{d-1}\right].\]
In particular, under the additional assumptions that
$A\le C'\sqrt{d}$ and $B=1$, then 
\[\E d_{BL}(X_\theta,\sigma Z)
\le C \frac{k+\log(d)}{k^{\frac{2}{3}}
d^{\frac{2}{3k+4}}}.\]
\end{thm}

\noindent {\bf Remark:} The assumption that $B=1$ is automatically
satisfied if the covariance matrix of $X$ is the identity; in the
language of convex geometry, this is simply the case that the 
vector $X$ is isotropic.  The assumption that $A=O(\sqrt{d})$
is a geometrically natural one which arises, for example, if $X$ is 
distributed uniformly on the isotropic dilate of the $\ell_1$ ball in 
$\R^d$.

\medskip

Together, Theorems \ref{concX} and \ref{dist-bdX} give the 
following.  

\begin{cor}\label{summary}  Let $X$ be a random vector in $\R^d$
satisfying \[\E|X|^2=\sigma^2d\qquad\E||X|^2\sigma^{-2}-d|\le L\sqrt{d} \qquad 
\sup_{\xi\in\s^{d-1}}\E\inprod{\xi}{X}^2\le 1.\]
Let  $X_\theta$ denote the projection of $X$ onto the span of $\theta$,
for $\theta\in\mathfrak{W}_{d,k}$.  Fix $a>0$ and $b<2$
and suppose 
that  
$k=\delta\frac{\log(d)}{\log(\log(d))}$ with $a\le\delta\le b$.
Then  there is a $c>0$ depending only on $a$ and $b$ such that
for \[\epsilon=2\exp\left[-c\frac{\log(\log(d))}{
\delta}\right],\]
there is a subset 
$\mathfrak{T}\subseteq\mathfrak{W}_{d,k}$ with $\P_{d,k}[\mathfrak{T}]\ge
1-C\exp\left(-c'd\epsilon^2\right)$, such that for all $\theta\in\mathfrak{T}$,
\[d_{BL}(X_\theta,\sigma Z)\le C'\epsilon.\]
\end{cor}

\smallskip

\noindent {\bf Remark:} For the bound on $\E d_{BL}(X_\theta,\sigma Z)$
given in \cite{JOTP} to tend to zero as $d\to\infty$, it is necessary that
$k=o(\sqrt{\log(d)})$, whereas Theorem \ref{dist-bdX} gives a
similar result if $k=\delta\left(\frac{\log(d)}{\log(\log(d))}\right)$
for $\delta<2$.  
Moreover, the following example shows that
the bound above is best possible in our metric.

\medskip

\subsection{Sharpness}

In the presence of log-concavity of the distribution of $X$, Klartag
\cite{Kla2} proved a stronger result than Corollary \ref{summary} above; namely,
that the typical total variation distance between $X_\theta$ and the
corresponding Gaussian distribution is small even when $\theta\in
\mathfrak{W}_{d,k}$ and $k=d^{\epsilon}$ (for a specific universal constant 
$\epsilon\in(0,1)$).  The result above allows $k$ to 
grow only a bit more slowly than 
logarithmically with $d$.  However, as the following example
shows, either the log-concavity or some other additional assumption is
necessary; with only the assumptions here, logarithmic-type growth
of $k$ in $d$ is best possible for the bounded-Lipschitz metric.  
(It should be noted that the specific
constants appearing in the results above are almost certainly non-optimal.)

Let $X$ be distributed uniformly among $\{\pm\sqrt{d}e_1,\ldots,\pm\sqrt{d}
e_d\}$, where the $e_i$ are the standard basis vectors of $\R^d$.  That
is, $X$ is uniformly distributed on the vertices of a cross-polytope.
Then $\E[X]=0$, $|X|^2\equiv d$,
and given $\xi\in\s^{d-1}$, $\E\inprod{X}{\xi}^2=1$; Theorems \ref{meanX},
\ref{concX} and \ref{dist-bdX} apply with $\sigma^2=1$, $A=0$ and $B=1$.

Consider a projection of $\{\pm\sqrt{d}e_1,\ldots,\pm\sqrt{d}
e_d\}$ onto a random subspace $E$ of dimension $k$, and
define the Lipschitz function $f:E\to\R$ by 
$f(x):=\left(1-d(x,S_E)\right)_+,$
where $S_E$ is the image of $\{\pm\sqrt{d}e_1,\ldots,\pm\sqrt{d}
e_d\}$ under projection onto $E$ and $d(x,S_E)$ denotes the (Euclidean)
distance from the point $x$ to the set $S_E$.  Then if $\mu_{S_E}$ denotes
the probability measure putting equal mass at each of the points of $S_E$, 
$\int fd\mu_{S_E}=1$.  On the other hand, it is classical
(see, e.g., \cite{Fol}) that the volume $\omega_k$ of the unit ball in
$\R^k$ is asymptotically given by $\frac{\sqrt{2}}{\sqrt{k\pi}}\left[
\frac{2\pi e}{k}\right]^{\frac{k}{2}}$ for large $k$, in the sense that the
ratio tends to one as $k$ tends to infinity.  It follows that 
the standard Gaussian measure of a ball 
of radius 1 in $\R^k$ is bounded by $\frac{1}{(2\pi)^{k/2}}\omega_k\sim
\frac{\sqrt{2}}{\sqrt{k\pi}}\left[\frac{e}{k}\right]^{\frac{k}{2}}$.
If $\gamma_k$ denotes the standard Gaussian measure in $\R^k$, then this
estimate means that $\int fd\gamma_k\le \frac{2\sqrt{2}d}{\sqrt{k\pi}}
\left[\frac{e}{k}\right]^{\frac{k}{2}}$.  
Now, if $k=\frac{c\log(d)}{\log(
\log(d))}$ for $c>2$, 
then this bound tends to zero, and thus $d_{BL}(\mu_{S_E},\gamma_k)$ is
close to 1 for any choice of the subspace $E$; the measures $\mu_{S_E}$
are far from Gaussian in this regime.  

Taken together with Corollary \ref{summary}, this shows that the 
phenomenon of typically Gaussian marginals persists for $k=\frac{c
\log(d)}{\log(\log(d))}$ for $c<2$, but fails in general if $k=\frac{c\log(d)}{\log(
\log(d))}$ for $c>2$.

Continuing the analogy with Dvoretzky's theorem, it is worth noting here
that, for the projection formulation of Dvoretzky's theorem (the dual 
viewpoint to the slicing version discussed above), the worst case behavior
is achieved for the $\ell_1$ ball, that is, for the convex hull of the points
considered above.

\subsection{Acknowledgements} The author thanks Mark Meckes for many useful
discussions, without which this paper may never have been completed.
Thanks also to Michel Talagrand, who pointed out a simplification in the 
proof of the main theorem.

\section{Proofs}
Theorems \ref{meanX} and \ref{concX} were proved in \cite{JOTP}, and their
proofs will not be reproduced.

This section is mainly devoted to the proof of Theorem \ref{dist-bdX},
but first some more definitions and notation are needed. Firstly, a
comment on distance: as is clear from the statement of Theorems
\ref{concX} and \ref{dist-bdX}, the metric on random variables used here
is the bounded-Lipschitz distance, defined by $d_{BL}(X,Y):=\sup_f\big|\E f(X)-
\E f(Y)\big|$, where the supremum is taken over functions $f$ with 
$\|f\|_{BL}:=\max\{\|f\|_\infty,|f|_L\}\le 1$ ($|f|_L$ is the Lipschitz constant
of $f$).  

\medskip

A centered stochastic process 
$\{X_t\}_{t\in T}$ indexed
by a space $T$ with a metric $d$ is said to satisfy a {\em sub-Gaussian
increment condition} if there is a constant $C$ such that, 
for all $\epsilon>0$,
\begin{equation}\label{subGauss}
\P\big[|X_s-X_t|\ge\epsilon\big]\le C\exp\left(-\frac{\epsilon^2}{2d^2(s,t)}
\right).
\end{equation}

A crucial point for the proof of Theorem \ref{dist-bdX} is that in the 
presence of a sub-Gaussian increment condition, there are powerful tools
availabe to bound the expected supremum of a stochastic 
process; the one used here is
the entropy bound of Dudley \cite{Dud}, formulated in terms of entropy
numbers {\em \`a la} Talagrand \cite{tal}.  
For $n\ge 1$, the {\em entropy number} $e_n(T,d)$ is defined by
$$e_n(T,d):=\inf\{\sup_td(t,T_n):T_n\subseteq T, |T_n|\le 2^{2^n}\}.$$
Dudley's entropy bound is the following.
\begin{thm}[Dudley]\label{ent-thm}
If $\{X_t\}_{t\in T}$ is a centered stochastic process satisfying the
sub-Gaussian increment condition \eqref{subGauss}, then there is a constant $L$
such that 
\begin{equation}\label{ent}
\E\left[\sup_{t\in T}X_t\right]\le L\sum_{n=0}^\infty 2^{n/2}e_n(T,d).
\end{equation}
\end{thm}

We now give the proof of the main theorem.

\begin{proof}[Proof of Theorem \ref{dist-bdX}]
As in \cite{JOTP}, the key initial step is to view the
distance as the supremum of a stochastic process: let
$X_f=X_f(\theta):=\E_Xf(X_\theta)- \E f(X_\Theta)$.  Then $\{X_f\}_f$ is a
centered stochastic process indexed by the unit ball of $\|\cdot\|_{BL}$, 
and $d_{BL}(X_\theta,X_{\Theta})=\sup_{\|f\|_{BL}\le 1}X_f$.  The fact
that Haar measure on $\mathfrak{W}_{d,k}$ has a measure-concentration 
property for Lipschitz functions (see \cite{MilSch})
implies that $X_f$ is a sub-Gaussian
process, as follows.

Let $f:\R^k\to\R$ be Lipschitz with Lipschitz constant $L$ and
 consider the 
 function $G=G_f$ defined on $\mathfrak{W}_{d,k}$ by 
 $$G(\theta_1,\ldots,\theta_k)=\E_X f(X_\theta)=
\E\left[f(\inprod{\theta_1}{X},\ldots,\inprod{\theta_k}{X})\big|\theta\right].$$
Then
\begin{equation*}\begin{split}
\Big|G(\theta)-G(\theta')\Big|&=
\left|\E\left[f\big(\inprod{X}{\theta_1'},\ldots,\inprod{X}{\theta_k'}\big)-
f\big(\inprod{X}{\theta_1},\ldots,\inprod{X}{\theta_k}\big)\Big|\theta,\theta'
\right]\right|\\&\le L\E\left[\big|\big(\inprod{X}{\theta_1'-\theta_1},\ldots,
\inprod{X}{\theta_k'-\theta_k}\big)\big|\Big|\theta,\theta'\right]
\\&\le L \sqrt{\sum_{j=1}^k|\theta_j'-\theta_j|^2\E\inprod{X}{\frac{
\theta_j'-\theta_j}{|\theta_j'-\theta_j|}}^2}\\&\le
L\rho(\theta,\theta')\sqrt{B},
\end{split}\end{equation*}
thus $G(\theta)$ is a Lipschitz function on 
$\mathfrak{W}_{k,d}$, with Lipschitz constant $L\sqrt{B}$.
It follows immediately
from Theorem 6.6 and remark 6.7.1 of \cite{MilSch} that 
$$\P_{d,k}\left[\left|G(\theta)-M_G\right|>\epsilon\right]\le \sqrt{\frac{\pi}{2}}
e^{-\frac{d\epsilon^2}{8L^2B}},$$
where $M_G$ is the median of $G$ with respect to Haar measure on 
$\mathfrak{W}_{d,k}$.  It is then a straightforward exercise to show that
for some universal constant $C$, 
\begin{equation}\begin{split}\label{GconcX}
\P\left[\left|G(\theta)-\E G(\theta)\right|>\epsilon\right]
&\le Ce^{-\frac{d\epsilon^2}{32L^2B}}.
\end{split}\end{equation}
Observe that, for $\Theta$ a Haar-distributed random point of $\mathfrak{W}_{
d,k}$, $\E G(\Theta)=\E f(X_\Theta)$,
and so \eqref{GconcX} can be restated as
$\P\left[|X_f|>\epsilon\right]\le
C\exp\left[-cd\epsilon^2\right].$

Note that 
$X_f-X_g=X_{f-g},$
thus for $|f-g|_L$ the Lipschitz
constant of $f-g$ and $\|f-g\|_{BL}$ the bounded-Lipschitz norm of $f-g$,
$$\P\left[\big|X_f-X_g\big|>\epsilon\right]\le C\exp\left[\frac{-cd\epsilon^2}{
2|f-g|_L^2}\right]\le C\exp\left[\frac{-cd\epsilon^2}{
2\|f-g\|_{BL}^2}\right].$$
The process $\{X_f\}$
therefore satisfies the sub-Gaussian increment condition in the metric 
$d^*(f,g):=\frac{1}{\sqrt{cd}}\|f-g\|_{BL}$;  in particular,
 the entropy bound \eqref{ent} applies.  We will not be able to 
apply it directly, but rather use a sequence of approximations to arrive
at a bound.

The first step is to truncate the indexing functions.  
Let 
$$\varphi_R(x)=\begin{cases}1&|x|\le R,\\R+1-|x|&R\le|x|\le R+1,\\0&R+1\le|x|,
\end{cases}$$
and define $f_R:=f\cdot\varphi_R$.  It is easy to see that if $\|f\|_{BL}\le 1$, then $\|f_R\|_{BL}
\le2$.  
Since $|f(x)-f_R(x)|=0$ if $x\in B_R$ and $|f(x)-f_R(x)|\le 1$ for 
all $x\in\R^k$, 
$$\big|\E_X f(X_\theta)-\E_X f_R(X_\theta)\big|\le\P\big[|X_\theta|>R\big|
\theta\big]\le
\frac{1}{R^2}\sum_{i=1}^k\E\big[
\inprod{X}{\theta_i}^2\big]\le\frac{Bk}{R^2},$$
and the same holds if $\E_X$ is replaced by $\E$.
It follows that 
$\left|X_f-X_{f_R}\right|\le\frac{2Bk}{R^2}.$
Consider therefore the process $X_f$ indexed by $BL_{2,R+1}$ (with norm
$\|\cdot\|_{BL}$), for some
choice of $R$ to be determined, where
$$BL_{2,R+1}:=\left\{f:\R^k\to\R:\|f\|_{BL}\le 2; f(x)=0\,{\rm if}\,|x|>R+1
\right\};$$
what has been shown is that 
\begin{equation}\label{trunc_error}
\E\Big[\sup_{\|f\|_{BL}\le 1}X_f\Big]\le\E\Big[\sup_{f\in BL_{2,R+1}}X_f
\Big]+\frac{2Bk}{R^2}.
\end{equation}

The next step is to approximate functions in $BL_{2,R+1}$ by
``piecewise linear'' functions.  Specifically, consider a cubic
lattice of edge length $\epsilon$ in $\R^k$.  Triangulate each cube of
the lattice into simplices inductively as follows: in $\R^2$, add an
extra vertex in the center of each square to divide the square into
four triangles.  To triangulate the cube of $\R^k$, first triangulate
each facet as was described in the previous stage of the induction.
Then add a new vertex at the center of the cube; connecting it to each 
of the vertices of each of the facets gives a triangulation into simplices.
Observe that when this procedure is carried out, each new vertex added is
on a cubic lattice of edge length $\frac{\epsilon}{2}$.  Let $\mathcal{L}$
denote the supplemented lattice comprised of the original cubic lattice,
together with the additional vertices needed for the triangulation.  The 
number of sites of $\mathcal{L}$ within the ball of radius $R+1$
is then bounded by, e.g., $c\left(\frac{3
R}{\epsilon}\right)^k\omega_k$, where $\omega_k$ is the volume of the unit
ball in $\R^k$.

Now approximate $f\in BL_{2,R+1}$ by the function $\tilde{f}$ defined
such that $\tilde{f}(x)=f(x)$ for $x\in\mathcal{L}$, and the graph
of $\tilde{f}$ is determined by taking the convex hull of the vertices
of the image under $f$ of each $k$-dimensional simplex determined by
$\mathcal{L}$.  The resulting function $\tilde{f}$ still has
$\|\tilde{f}\|_{BL}\le 2$, and $\|f-\tilde{f}\|_\infty\le\frac{\epsilon
\sqrt{k}}{2}$,
since the distance between points in the same simplex is bounded by
$\epsilon\sqrt{k}$.  Moreover, $\|\tilde{f}\|_{BL}=\sup_{x\in
\mathcal{L}}|f(x)|+\sup_{x\sim y}\frac{|f(x)-f(y)|}{|x-y|}$,
where $x\sim y$ if $x,y\in\mathcal{L}$ and $x$ and $y$ are part of the 
same triangulating simplex.
Observe that, for a given $x\in\mathcal{L}$, those vertices which are 
part of a triangulating simplex with $x$ are all contained 
 in a cube centered at $x$
of edge length $\epsilon$; the number of such points is thus bounded by
$3^k$, and the number of differences which must be considered in order to
compute the Lipschitz constant of $\tilde{f}$ is therefore bounded by
$c\left(\frac{9
R}{\epsilon}\right)^k\omega_k$.  Recall that $\omega_k\sim\frac{2}{\sqrt{k\pi}
}\left[
\frac{2\pi e}{k}\right]^{\frac{k}{2}}$ for large $k$, and so the number 
of differences determining the Lipschitz constant of $\tilde{f}$ is bounded
by $\frac{c}{\sqrt{k}}\left(\frac{c'R}{\epsilon\sqrt{k}}\right)^k$, for 
some absolute constants $c,c'$.  It follows that
\begin{equation}\label{linearized}
\E\Big[\sup_{f\in BL_{2,R+1}}X_f\Big]\le\E\Big[\sup_{f\in BL_{2,R+1}}X_{\tilde{f}}
\Big]+\epsilon\sqrt{k},
\end{equation}
that the process $\{X_{\tilde{f}}\}_{f\in BL_{2,R+1}}$ is sub-Gaussian with respect
to $\frac{1}{\sqrt{cd}}\|\cdot\|_{BL}$, 
and that the values of $\tilde{f}$ for $f\in BL_{2,R+1}$ are determined
by a point of the ball $2B_\infty^M$ of $\ell_\infty^M$, where 
\begin{equation}\label{Mbound}
M= \frac{c}{\sqrt{k}}\left(\frac{c'R}{\epsilon\sqrt{k}}\right)^k.
\end{equation}

The virtue of this approximation is that it replaces a sub-Gaussian process
indexed by a ball in an infinite-dimensional
space with one indexed by a ball in a finite-dimensional space, where 
Dudley's bound is finally to be applied.  
Let $T:=\left\{\tilde{f}:f\in BL_{2,R+1}\right\}\subseteq 2B_\infty^M$;
the covering numbers
of the unit ball $B$ of a finite-dimensional normed space $(X,\|\cdot\|)$
of dimension $M$ 
are known (see Lemma 2.6 of \cite{MilSch}) to be bounded as
$\mathcal{N}(B,\|\cdot\|,\epsilon)\le \exp\left[M\log\left(\frac{3}{
\epsilon}\right)\right].$
This implies that 
$$\mathcal{N}(B^M_\infty,\rho,\epsilon)\le \exp\left[M\log\left(\frac{3
}{\epsilon\sqrt{cd}}\right)\right],$$
which in turn implies that
$$e_n(2B_\infty^M,\rho)\le\frac{24\sqrt{B}}{\sqrt{d}}
2^{-\frac{2^n}{M}}.$$
Applying Theorem \ref{ent-thm} now yields
\begin{equation}\label{ent-est}
\E\left[\sup_{f\in BL_{2,R+1}}X_{\tilde{f}}\right]\le L\sum_{n\ge 0}
\left(\frac{24\sqrt{B}}{\sqrt{d}}2^{
\left(\frac{n}{2}-\frac{2^n}{M}\right)}\right).
\end{equation}
Now, for the terms in the sum with $\log(M)\le (n+1)\log(2)-3\log(n)$, 
the summands are bounded above by $2^{-n}$, contributing only a constant
to the upper bound.  On the other hand, the summand is maximized for 
$2^n=\frac{M}{2}\log(2)$, and is therefore bounded by $\sqrt{M}$.
Taken together, these estimates show that the sum on the right-hand
side of \eqref{ent-est} is bounded by $L\log(M)\sqrt{\frac{MB}{d}}$.

Putting all the pieces together, 
$$\E\left[\sup_{\|f\|_{BL}\le 1}\left(\E\left[f(X_\Theta)\big|\Theta\right]-
    \E f(X_\Theta)\right)\right]\le
\frac{9kB}{R^2}+2\epsilon\sqrt{k}+L\log(M)\sqrt{\frac{ MB}{d}}.$$
Choosing $\epsilon=\frac{\sqrt{k}B}{2R^2}$ and using the value of $M$
in terms of $R$ yields
$$\E\left[\sup_{\|f\|_{BL}\le 1}\left(\E\left[f(X_\Theta)\big|\Theta\right]-
\E f(X_\Theta)
\right)\right]\le \frac{10kB}{R^2}+Lk\log\left(\frac{c' R^3}{
kB}\right)\frac{c}{k^{1/4}}\left[\frac{c' R^3}{kB}\right]^{
\frac{k}{2}}\sqrt{\frac{
B}{d}}.$$
Now choosing $R=cd^{\frac{1}{3k+4}}k^{\frac{2k+1}{6k+8}}B^{\frac{k+1}{3k+4}}$
yields
\[\E\left[\sup_{\|f\|_{BL}\le 1}\left(\E\left[f(X_\Theta)\big|\Theta\right]-
\E f(X_\Theta)
\right)\right]\le L\frac{kB+B\log(d)}{d^{\frac{2}{3k+4}}
k^{\frac{2k+1}{3k+4}}B^{\frac{2k+2}{3k+4}}}.\]
This completes the proof of the first statement of the theorem.  
The second follows immediately using that $B=1$ and observing that, 
under the assumption that 
$A\le C'\sqrt{d}$, the 
bound above is always worse than the error
$\frac{\sigma[\sqrt{k}(A+1)+ k]}{d-1}$ 
coming from Theorem \ref{meanX}.

\end{proof}

The proof of Corollary \ref{summary} is essentially immediate from
Theorems \ref{concX} and \ref{dist-bdX}.

\bibliographystyle{plain}
\bibliography{sharp_projections}

\end{document}